\documentclass[11pt, leqno]{amsart}

\usepackage[svgnames, dvipsnames, table]{xcolor}
\usepackage[colorlinks = true, allcolors = UCRceleste, pagebackref=true, colorlinks]{hyperref}

\usepackage{float}
\usepackage[utf8]{inputenc}
\usepackage{caption}
\usepackage{subcaption}
\usepackage{float}
\usepackage[ruled, linesnumbered, algosection]{algorithm2e}
\usepackage{mathrsfs}
\usepackage{verbatim}
\usepackage{xcolor}
\usepackage{color}
\usepackage{graphicx}
\usepackage{rotating}
\usepackage{diagbox}
\usepackage{amssymb}
\usepackage{epstopdf}
\usepackage{tikz}
\definecolor{UCRceleste}{RGB}{0,192,243}
\definecolor{mintgreen}{RGB}{152,255,152}
\definecolor{pinksalmon}{RGB}{255,102,102}
\definecolor{hueso}{RGB}{245,245,220}
\definecolor{marfil}{RGB}{255,253,208}
\definecolor{amarillo}{RGB}{255,255,0}
\usetikzlibrary{decorations.markings,arrows}

\usetikzlibrary{decorations.pathreplacing}
\usepackage[inner=1.0in,outer=1.0in,bottom=1.0in, top=1.0in]{geometry}

\usepackage{enumitem}

\usepackage{orcidlink}

\numberwithin{equation}{section}

\newtheorem{theorem}{Theorem}[section]

\newtheorem{lemma}[theorem]{Lemma}
\newtheorem{proposition}[theorem]{Proposition}

\theoremstyle{plain}
\newtheorem{definition}[theorem]{Definition}

\theoremstyle{remark}
\newtheorem{remark}[theorem]{Remark}

\makeatletter
\def\moverlay{\mathpalette\mov@rlay}
\def\mov@rlay#1#2{\leavevmode\vtop{%
   \baselineskip\z@skip \lineskiplimit-\maxdimen
   \ialign{\hfil$\m@th#1##$\hfil\cr#2\crcr}}}
\newcommand{\charfusion}[3][\mathord]{
    #1{\ifx#1\mathop\vphantom{#2}\fi
        \mathpalette\mov@rlay{#2\cr#3}
      }
    \ifx#1\mathop\expandafter\displaylimits\fi}
\makeatother

\newcommand{\suchthat}{\;\ifnum\currentgrouptype=16 \middle\fi|\;}

\newcommand{\Q}{\mathbb{Q}}

\newcommand{\op}[1]{\operatorname{#1}}

\usepackage[svgnames, dvipsnames, table]{xcolor}
\usepackage[colorlinks = true, allcolors = UCRceleste, pagebackref=true, colorlinks]{hyperref}

\usepackage[utf8]{inputenc}
\usepackage{caption}
\usepackage{subcaption}
\usepackage{float}
\usepackage{mathrsfs}
\usepackage{xcolor}
\usepackage{mathtools}
\usepackage{seqsplit}

\usepackage{color}
\usepackage{amssymb}
\usepackage{tikz}
\definecolor{UCRceleste}{RGB}{0,192,243}
\definecolor{mintgreen}{RGB}{152,255,152}
\definecolor{pinksalmon}{RGB}{255,102,102}
\definecolor{hueso}{RGB}{245,245,220}
\definecolor{marfil}{RGB}{255,253,208}
\definecolor{amarillo}{RGB}{255,255,0}

\usepackage{geometry}

\usepackage{enumitem}

\usepackage{orcidlink}

\usepackage[ruled, linesnumbered, algosection]{algorithm2e}

\numberwithin{equation}{section}

\newtheorem*{theorem*}{Theorem}

\newcommand{\val}{\operatorname{val}}

\newcommand{\legendre}[2]{\ensuremath{\left( \frac{#1}{#2} \right) }}
\newcommand{\tlegendre}[2]{\ensuremath{( \frac{#1}{#2} ) }}

\newcommand{\krop}{{\tlegendre{\cdot}{p}}}

\def\O_K{{\Cal{O}_{K}}}
\def\O_F{{\Cal{O}_{F}}}
\def\N_F{{\Cal{N}_{F/\Q}}}

\begin{document}
\raggedbottom

\title{Efficient Computation and Congruences for Colored Partition Functions
}

\author[J.C. Villegas-Morales]{Jean Carlos Villegas-Morales\orcidlink{0009-0004-8933-0627}}

\address{Department of Mathematics, 3368 TAMU, Texas A\&M University, College Station, TX 77843-3368, USA}
\email{jean\_villegas.02@tamu.edu}

\begin{abstract}
For a positive integer $\alpha$, let $p_\alpha(n)$ denote the number of $\alpha$-colored partitions of $n$. The Rademacher-type expansion of Iskander, Jain, and Talvola is valid for every real $\alpha>0$, but when $\alpha>24$ it involves several polar terms and a two-parameter family of exponential sums $A_k^{(\alpha)}(n,m)$. For integral $\alpha$, we prove multiplicativity and prime-power reduction formulas for these sums, expressing their local factors as classical or quadratically twisted Kloosterman sums. Together with explicit truncation and precision bounds, these formulas yield an efficient algorithm for computing $p_\alpha(n)$ exactly. Our SageMath implementation computes the 1,113,767-digit integer $p_{100}(10^{10})$ in less than one hour. As an application, we use the algorithm to certify new Ramanujan-type congruences.
\end{abstract}

\maketitle

\section{Introduction}

The ordinary partition function \(p(n)\) is one of the central objects in additive number theory. Ramanujan's celebrated congruences
\[
p(5n+4)\equiv0\pmod 5,\qquad
p(7n+5)\equiv0\pmod 7,\qquad
p(11n+6)\equiv0\pmod{11}
\]
initiated the systematic study of partition congruences \cite{MR1544457}. For a positive integer \(\alpha\), the \(\alpha\)-colored partition function \(p_\alpha(n)\) counts partitions in which each part may be assigned one of \(\alpha\) colors. Its generating function is
\[
\sum_{n\geq0}p_\alpha(n)q^n=\prod_{n\geq1}(1-q^n)^{-\alpha},
\]
and therefore
\[\eta(z)^{-\alpha}=q^{-\alpha/24}\sum_{n\geq0}p_\alpha(n)q^n,\qquad q=e^{2\pi i z}.
\]
Thus, colored partition numbers occur naturally as Fourier coefficients of negative-weight eta-quotients.

The classical partition function can be computed using the convergent Hardy--Ramanujan--Rademacher series. More generally, Sussman \cite[Theorem~1.1]{Sus17} obtained Rademacher-type expansions for a broad class of negative-weight eta-quotients. The associated exponential sums were analyzed in terms of twisted Kloosterman sums in \cite{AJNJ}, extending classical work of Lehmer \cite{Leh38} and Sali\'e \cite{Sal32}. For the eta-quotient \(\eta(z)^{-\alpha}\), however, condition \cite[(15)]{Sus17} restricts this approach to \(0<\alpha\leq24\).

Iskander, Jain, and Talvola \cite[Theorem~1.1]{IJT20} removed this restriction by proving a Hardy--Ramanujan--Rademacher-type expansion for \(p_\alpha(n)\) valid for every real \(\alpha>0\). Their formula involves exponential sums \(A_k^{(\alpha)}(n,m)\) arising from the polar terms of \(\eta(z)^{-\alpha}\). The sum previously studied in \cite{AJNJ} for the colored partition function is precisely \(A_k^{(\alpha)}(n,0)\). When \(\alpha>24\), additional terms with \(m>0\) occur, and both arguments of the corresponding Kloosterman sums vary. Consequently, the one-parameter formulas do not suffice to evaluate the full expansion efficiently.

For positive integral \(\alpha\), we establish the required arithmetic formulas for the full two-parameter family \(A_k^{(\alpha)}(n,m)\). We first prove a multiplicativity formula that reduces their evaluation to prime-power moduli (Theorem~\ref{thm:multiplicativity}). We then identify the prime-power factors with classical or quadratically twisted Kloosterman sums (Theorem~\ref{prop:Ak-prime-power}) and give the reductions needed for their explicit evaluation. These results constitute the principal arithmetic contribution of the paper and make the Iskander--Jain--Talvola expansion computationally effective when \(\alpha>24\).

Combining these formulas with explicit truncation estimates and an adaptation of Johansson's termwise precision strategy \cite{Joh12}, we obtain a rigorous procedure for computing individual values of \(p_\alpha(n)\) (Algorithm~\ref{alg:colored-partitions}). Each term is evaluated at a precision determined by an explicit bound for its magnitude, while the prime-power formulas detect many vanishing contributions before numerical evaluation. As a large example, our SageMath implementation computes \(p_{100}(10^{10})\), an integer with more than one million digits, in less than one hour.

Congruences for colored partition functions have been studied by Gandhi \cite{gandhi1963congruences}, Atkin \cite{Atkin1968}, Kiming and Olsson \cite{KimingOlsson}, Boylan \cite{Boylan}, Garvan \cite{MR1055413}, Chen, Du, Hou, and Sun \cite{ChenDuHouSun}, Dicks \cite{Dicks}, Tang \cite{MR3766907}, and Beckwith et al.\ \cite{MR4624922}, among others. Building on the existence results of Treneer \cite{Tre08}, Ryan, Scherr, Sirolli, and Treneer gave an effective finite criterion for certifying congruences of weakly holomorphic modular forms and eta-quotients \cite[Theorems~4.2 and~4.3]{RSST21}. We specialize their criterion to colored partition functions and combine it with our exact coefficient-computation procedure, obtaining the new explicit congruence families recorded in Theorem~\ref{thm:colored-congruences}.

The paper is organized as follows. Section~\ref{sect:HRR} recalls the Rademacher expansion. Sections~\ref{multisec} and ~\ref{pripowsec} establish the multiplicativity and prime-power reduction formulas for \(A_k^{(\alpha)}(n,m)\). Section~\ref{KloosEval} collects the known local Kloosterman-sum evaluations used to make these formulas explicit. Sections~\ref{sect:recursive}--\ref{sect:algorithms} present the recurrence relation, the coefficient and truncation bounds, and the exact computational procedure. Finally, Section~\ref{seccongru} applies this procedure to the
certification of colored-partition congruences.

\section*{Acknowledgments}

The author is deeply grateful to Nicol\'as Sirolli for valuable discussions and insightful comments on the mathematical and computational aspects of this work. The author also thanks Paul Dessauer for carefully reading the manuscript and suggesting improvements to its exposition.

ChatGPT was used for minor assistance with writing and code. All mathematical work and verification were carried out by the author, who takes full responsibility for the paper.

\section{The Rademacher expansion for colored partitions}
\label{sect:HRR}

We begin by fixing the notation for the Rademacher expansion used throughout the paper. Let \(z\in\mathbb{H}:=\{z\in\mathbb{C}:\operatorname{Im}(z)>0\}\). The Dedekind eta function is defined by
\begin{equation}
\eta(z)
=
q^{1/24}\prod_{n=1}^{\infty}(1-q^n).
\end{equation}

For every real \(\alpha>0\), define the coefficients \(p_\alpha(n)\) by
\begin{equation}
\sum_{n=0}^{\infty}p_\alpha(n)q^n
=
\prod_{n=1}^{\infty}\frac{1}{(1-q^n)^\alpha}.
\end{equation}
Equivalently,
\begin{equation}
\eta(z)^{-\alpha}
=
q^{-\alpha/24}
\sum_{n=0}^{\infty}p_\alpha(n)q^n.
\end{equation}
When \(\alpha\) is a positive integer, \(p_\alpha(n)\) counts the number of partitions of \(n\) in which each part is assigned one of \(\alpha\) possible colors.

To state the expansion, for real $\alpha>0$ and integers $n>\alpha/24$ and $0\leq m\leq\alpha/24$, set
\begin{equation}\label{eq:nu-mu}
s_\alpha
:=
\frac{\alpha}{2}+1,
\qquad
\nu_\alpha(n)
:=
\sqrt{n-\frac{\alpha}{24}},
\qquad
\mu_\alpha(m)
:=
\sqrt{\frac{\alpha}{24}-m},
\qquad\text{and}\qquad
M_\alpha
:=
\left\lfloor\frac{\alpha}{24}\right\rfloor.
\end{equation}

\begin{definition}
Let $\alpha>0$ be real, let $n,m\in\mathbb Z$, and let $k\geq1$ be an integer. The $\alpha$-Kloosterman sum is
\begin{equation}\label{eq:alpha-kloosterman}
A_k^{(\alpha)}(n,m)
:=
\sum_{\substack{0\leq h<k\\(h,k)=1}}
\exp\!\left(
\alpha\pi i\,s(h,k)
-
\frac{2\pi i}{k}(mh^{-1}+nh)
\right),
\end{equation}
where $h^{-1}$ denotes an inverse of $h$ modulo $k$, and
\begin{equation}
s(h,k)
:=\sum_{r=1}^{k-1}\frac{r}{k}\left(\frac{hr}{k}-\left\lfloor\frac{hr}{k}\right\rfloor-\frac12\right)
\end{equation}
is the Dedekind sum.
\end{definition}

We use the modified Bessel function of the first kind, defined by
\begin{equation}
I_\nu(x):=\left(\frac{x}{2}\right)^\nu\sum_{j=0}^{\infty}\frac{(x/2)^{2j}}{j!\,\Gamma(\nu+j+1)}.
\end{equation}

The following Rademacher-type expansion is due to Iskander, Jain, and Talvola \cite[Theorem~1.1]{IJT20}.

\begin{theorem}\label{thm:exact-formula}
For every real \(\alpha>0\) and every integer \(n>\alpha/24\), we have
\begin{equation}\label{eq:exact-formula}
p_\alpha(n)
=
\nu_\alpha(n)^{-s_\alpha}
\sum_{m=0}^{M_\alpha}
\mu_\alpha(m)^{s_\alpha}p_\alpha(m)
\sum_{k=1}^{\infty}
\frac{2\pi}{k}
A_k^{(\alpha)}(n,m)
I_{s_\alpha}\!\left(
\frac{4\pi}{k}
\nu_\alpha(n)\mu_\alpha(m)
\right).
\end{equation}
\end{theorem}

Theorem~\ref{thm:exact-formula} is valid for every real \(\alpha>0\). Beginning with the next section, however, we assume that \(\alpha\) is a positive integer, as required by the arithmetic and character-theoretic properties of the sums \(A_k^{(\alpha)}(n,m)\).

\section{Multiplicativity}\label{multisec}
We now begin the arithmetic analysis of \(A_k^{(\alpha)}(n,m)\). The first step is to factor a sum with composite modulus into a product of sums with relatively prime moduli. The additional parameter \(m\), which multiplies \(h^{-1}\) in the exponential, must be carried through this factorization. Our starting point is the following congruence for Dedekind sums, due to Rademacher and Whiteman \cite[Theorem~20]{RW41}.
\begin{proposition}\label{Rademacher-Whiteman}
Let $a, b, c$ be pairwise coprime positive integers. Then
\begin{align*}
\left( s(ab, c) - \frac{ab}{12c} \right) + \left( s(bc, a) - \frac{bc}{12a} \right) - \left( s(b, ac) - \frac{b}{12ac} \right)+\frac{abc}{12} \in 2\mathbb{Z}.
\end{align*}
\end{proposition}

Using the preceding congruence, we obtain the following multiplicativity formula for the $\alpha$-Kloosterman sums.

\begin{theorem}\label{thm:multiplicativity}Let $k_1,k_2 \geq 1$ be relatively prime integers, and let $k= k_1 k_2$. Let $\theta_1,\theta_2$ be positive integers such that $\theta_1 \theta_2 = 24$ and $\gcd(\theta_1 k_1, \theta_2 k_2) = 1$. Then there exist $n_1,n_2\in\mathbb{Z}_{\geq 1}$ which are solutions to
\begin{equation}
    \label{eqn:mult_n1n2}
\begin{cases}(24n+\alpha u(k_1,k_2))\equiv 24n_1k_2^2\pmod{\theta_1k_1},\\
 (24n+\alpha u(k_1,k_2))\equiv 24n_2k_1^2\pmod{\theta_2k_2},\end{cases}
\end{equation}
where $u(k_1,k_2) = k_1^2 +k_2^2-k^2-1$. Moreover, they satisfy $A_k^{(\alpha)}(n,m)=A_{k_1}^{(\alpha)}(n_1,m_1)A_{k_2}^{(\alpha)}(n_2,m_2)$ where $m_1,m_2$ are integers satisfying $m \equiv m_1 \pmod{k_1}$ and $ m\equiv m_2 \pmod{k_2}$.
\end{theorem}

\begin{proof}
The proof adapts the factorization argument of \cite[Theorem 5.1]{AJNJ}. If \(k_1=1\) or \(k_2=1\), the result follows immediately from \(A_1^{(\alpha)}(n,m)=1\). Hence we may assume that \(k_1,k_2>1\). To show that there exist $n_1,n_2\in\mathbb{Z}_{\geq 1}$ solving the linear system of congruences in the statement it suffices to see that both
$\gcd\left(24k_2^2,\theta_1k_1\right)$ 
and
$\gcd\left(24k_1^2,\theta_2k_2\right)$ 
divide
\(24n+\alpha u(k_1,k_2)\).

This follows immediately from the fact that they both divide $24$, and the fact that if $s,t\in\mathbb{Z}$ are relatively prime then $24 \mid s^2+t^2-s^2t^2-1$.

Since \(\gcd(k_1,k_2)=1\), multiplication by \(k_j\) permutes \((\mathbb Z/k_i\mathbb Z)^\times\) whenever \(i\neq j\). Therefore we can write
\begin{equation*}
	A_{k_i}^{(\alpha)}\left(n_i,m_i\right)  = 
	\sum_{\substack{0\leq h_i<k_i \\(h_i, k_i)=1}} \exp
\left[\pi I\left(\alpha s\left(k_jh_i,k_i\right)-\frac{2(m_ik_j^{-1}h_i^{-1}+n_ik_jh_i)}{k_i}\right)\right], \quad i \neq j,
\end{equation*}
where \(I\) temporarily denotes the imaginary unit, and \(h_i^{-1}\) and \(k_j^{-1}\) denote the inverses of \(h_i\) and \(k_j\) modulo \(k_i\), respectively. Then by the Chinese Remainder Theorem we get that

\begin{multline*}
A_{k_1}^{(\alpha)}(n_1,m_1) A_{k_2}^{(\alpha)}(n_2,m_2) = \sum_{\substack{0\leq h<k \\(h, k)=1}}\exp \left[\pi I\left(-\frac{2(m_1k_2^{-1}h^{-1}+n_1k_2h)}{k_1}-\right.\right.\\\left.\left.
\frac{2(m_2k_1^{-1}h^{-1}+n_2k_1h)}{k_2}+\alpha s(k_2h,k_1)+\alpha s(k_1h,k_2)\right)\right].
\end{multline*}

It therefore suffices to show that

\begin{multline}\label{sufficientcondition}
-\frac{2(m_1k_2^{-1}h^{-1}+n_1k_2h)}{k_1}
-\frac{2(m_2k_1^{-1}h^{-1}+n_2k_1h)}{k_2}+
\\
\frac{2(mh^{-1}+nh)}{k_1 k_2}
+\alpha s(k_2h,k_1)+\alpha s(k_1h,k_2)-\alpha s(h,k_1k_2)
\in 2\mathbb{Z}
\end{multline}
for every $h$ prime to $k$, because this yields $A_k^{(\alpha)}(n,m)=A_{k_1}^{(\alpha)}(n_1,m_1)A_{k_2}^{(\alpha)}(n_2,m_2)$.
Fix $h$ as above. First, we write $v = k_1^{-1}$ and $u = k_2^{-1}$. Then we already have
\[
-\frac{2m_1k_2^{-1}h^{-1}}{k_1}-
\frac{2m_2k_1^{-1}h^{-1}}{k_2}+
\frac{2mh^{-1}}{k_1 k_2} = \frac{2h^{-1}}{k_1k_2}\left(-m_1k_2u-m_2k_1v+m\right)\in 2\mathbb{Z},
\]
since by hypothesis and by the Chinese Remainder Theorem we have $m \equiv m_1k_2u+m_2k_1v \pmod{k_1k_2}$.
Second, note that
\[
    \gcd\left(k_1,k_2\right) = 
    \gcd\left(h,k_1\right) =
    \gcd\left(h,k_2\right)
    = 1.
    \]
We can therefore apply Proposition~\ref{Rademacher-Whiteman} with $a:=k_2,b:= h, c:=k_1$ to get that

\begin{align*}
\left( s(hk_2, k_1) - \frac{k_2h}{12k_1} \right) + \left( s(hk_1, k_2) - \frac{hk_1}{12k_2} \right) - \left( s(h, k_1k_2) - \frac{h}{12k_1k_2} \right)+\frac{k_1k_2h}{12} \in 2\mathbb{Z}.
\end{align*}
This gives that
\begin{equation}
    \label{eqn:true_hyp}
    -\frac{h \, u(k_1,k_2)}{12 k_1 k_2}
    +s(k_2h,k_1)+s(k_1h,k_2)-s(h,k_1k_2)
    \in 2\mathbb{Z}.
\end{equation}
After multiplying by $\alpha\in\mathbb{Z}$, it follows that \eqref{sufficientcondition} holds if and only if
\begin{align*}\frac{h}{12k_1k_2}\left(24n_1k_2^2+24n_2k_1^2-(24n+\alpha u(k_1,k_2))\right)\in 2\mathbb{Z}.\end{align*}
This condition can be made independent of $h$ by showing that\[24k_1k_2 = (\theta_1 k_1)(\theta_2 k_2)\mid
24n_1k_2^2+24n_2k_1^2-(24n+\alpha u(k_1,k_2)).\]

Since $\theta_jk_j\mid 24k_j^2$ for each $j\in\left\{1,2\right\}$, the above becomes equivalent to
\begin{align*}
	\theta_1k_1 & \mid 24n_1k_2^2-(24n+\alpha u(k_1,k_2))\\
	\theta_2k_2 & \mid 24n_2k_1^2-(24n+\alpha u(k_1,k_2)),
\end{align*}
which are true by definition of $n_1,n_2$.\end{proof}
Iterating Theorem~\ref{thm:multiplicativity} reduces the computation to prime-power moduli. We determine those local factors in the next section.

\section{Prime powers}\label{pripowsec}
As shown in the previous section, the multiplicativity formula reduces the computation of $A_k^{(\alpha)}(n,m)$ to prime-power moduli. We now express the resulting factors in terms of classical or quadratically twisted Kloosterman sums.
\begin{definition}\label{def:twisted-kloosterman}
Let $k\geq1$, let $a,b\in\mathbb Z$, and let $\chi$ be a Dirichlet character modulo $k$. The twisted Kloosterman sum associated with $\chi$ is
\[
S_\chi(a,b;k):=\sum_{\substack{0\leq h<k\\(h,k)=1}}\chi(h)\exp\left(\frac{2\pi i}{k}\left(ah+bh^{-1}\right)\right).
\]
When $\chi$ is trivial, this is the classical Kloosterman sum.
\end{definition} 
With $\omega(h,k):=e^{\pi i s(h,k)}$, equation~\eqref{eq:alpha-kloosterman} becomes
\begin{equation}
\label{eq:Ak-omega}
A_k^{(\alpha)}(n,m) = \sum_{\substack{0\le h<k\\(h,k)=1}} \omega(h,k)^{\alpha} \exp\!\left(-\frac{2\pi i}{k}(mh^{-1}+nh)\right).
\end{equation}

The multiplier \(\omega(h,k)\) admits the following explicit description due to Lehmer \cite[(1.4), (1.5)]{Leh38}:
\begin{equation}
\label{eq:lehmer-omega}
\omega(h,k)=
\begin{cases}
\displaystyle
\left(\frac{-h}{k}\right)(-i)^{(k-1)/2}
\exp\!\left(\frac{\pi i}{12k}f(h,k)\right),
& k \text{ odd},\\[1em]
\displaystyle
\left(\frac{-k}{h}\right)(-i)
\exp\!\left(\frac{\pi i}{12k}f(h,k)\right),
& k \text{ even},
\end{cases}
\end{equation}
where
\begin{equation}
\label{eq:f-hk}
f(h,k)
:=
\begin{cases}
\displaystyle
-(k^2-1)\left(2h+h^{-1}-h^2h^{-1}\right),
& k \text{ odd},\\[0.8em]
\displaystyle
h(k+1)(k+2)+(k^2-1)(h^2-1)h^{-1},
& k \text{ even}.
\end{cases}
\end{equation}

In order to rewrite \eqref{eq:Ak-omega} as a twisted Kloosterman sum, it is enough to determine \(f(h,k)\) modulo \(24k\).  We record the needed congruences in the following lemma.
\begin{lemma}
\label{lem:f-congruences}
Let $k\geq1$ and let $h$ be an integer with $(h,k)=1$. The function $f(h,k)$ defined in~\eqref{eq:f-hk} satisfies the following congruences. In each congruence, $h^{-1}$ denotes an inverse of $h$ modulo the indicated modulus.
\begin{enumerate}
\item[(i)] If $k$ is odd, then
\[
f(h,k)\equiv h+h^{-1}\pmod{k\text{ or }3k},
\]
where the modulus is $k$ if $3\nmid k$ and $3k$ if $3\mid k$. If $3\nmid k$, then also
\[
f(h,k)\equiv0\pmod3.
\]
For every odd $k$,
\[
f(h,k)\equiv0\pmod8.
\]
\item[(ii)] If $k=2^\lambda k_1$, where $\lambda\geq1$ and $k_1$ is odd, then
\[
f(h,k)\equiv h+h^{-1}\pmod{k_1\text{ or }3k_1},
\]
where the modulus is $k_1$ if $3\nmid k_1$ and $3k_1$ if $3\mid k_1$. If $3\nmid k_1$, then also
\[
f(h,k)\equiv0\pmod3.
\]
For every odd $k_1$, we also have
\[
f(h,k)\equiv h+h^{-1}+k(k+3)h\pmod{2^{\lambda+3}}.
\]
\end{enumerate}
\end{lemma}

\begin{proof}
The congruences follow by reducing the defining expressions for $f(h,k)$ modulo the indicated moduli, using the corresponding modular inverse relations and the fact that $x^2\equiv1\pmod8$ for every odd integer $x$.
\end{proof}

The preceding congruences hold for arbitrary $k$. We now specialize to prime-power moduli and apply the Chinese remainder theorem to obtain the following description of $f(h,k)$ modulo $24k$.

\begin{lemma}
\label{lem:f-prime-power}
Let $k=p^\lambda$, where $p$ is prime, and let $h$ be an integer with $(h,k)=1$. Then
\[
f(h,k)\equiv c \cdot t(h)\pmod{24k},
\]
where \(c\) and \(t(h)\) are given by
\[
\renewcommand{\arraystretch}{1.8}
\begin{array}{|c|c|c|}
\hline
p & c & t(h) \\
\hline
p>3
&
24
&
24^{-1}(h+h^{-1}) \pmod{k}
\\
\hline
p=3
&
8
&
8^{-1}(h+h^{-1}) \pmod{3k}
\\
\hline
p=2
&
3
&
3^{-1}\bigl(h+h^{-1}+k(k+3)h\bigr) \pmod{8k}
\\
\hline
\end{array}
\]
Here \(24^{-1}\), \(8^{-1}\), and \(3^{-1}\) denote the multiplicative inverses of \(24\), \(8\), and \(3\) modulo \(k\), \(3k\), and \(8k\),
respectively.
\end{lemma}

\begin{proof}
Apply the Chinese remainder theorem to the congruences in Lemma~\ref{lem:f-congruences}, considering separately $p>3$, $p=3$, and $p=2$.
\end{proof}
Combining Lemma~\ref{lem:f-prime-power} with \eqref{eq:Ak-omega} and \eqref{eq:lehmer-omega} yields the following prime-power formulas, which constitute the second main arithmetic ingredient of the algorithm.
\begin{theorem}
\label{prop:Ak-prime-power}
Let \(k=p^\lambda\), where \(p\) is prime. Define
\[
\chi_1(h):=\left(\frac{-h}{k}\right)^\alpha,
\qquad
\chi_2(h):=\left(\frac{-k}{h}\right)^\alpha.
\]
Then
\[
A_k^{(\alpha)}(n,m)=C\cdot\,S_{\chi}(a,b;k_p),
\]
where the parameters are given in the following table:
\[
\renewcommand{\arraystretch}{1.8}
\begin{array}{|c|c|c|c|c|c|}
\hline
p & k_p & C & \chi & a & b \\
\hline
p>3
&
k
&
(-i)^{\alpha(k-1)/2}
&
\chi_1
&
24^{-1}\alpha-n
&
24^{-1}\alpha-m
\\
\hline
p=3
&
3k
&
\frac{1}{3}(-i)^{\alpha(k-1)/2}
&
\chi_1
&
8^{-1}\alpha-3n
&
8^{-1}\alpha-3m
\\
\hline
p=2
&
8k
&
\frac{1}{8}(-i)^\alpha
&
\chi_2
&
3^{-1}\alpha(k^2+3k+1)-8n
&
3^{-1}\alpha-8m
\\
\hline
\end{array}
\]
Here the inverses \(24^{-1}\), \(8^{-1}\), and \(3^{-1}\) are taken modulo
\(k_p\) in their respective cases.
\end{theorem}

\begin{proof}
We give the details for \(p=2\), the other cases being analogous.
Let \(k=2^\lambda\). By equation \eqref{eq:lehmer-omega},
\[
\omega(h,k)^\alpha
=
\left(\frac{-k}{h}\right)^\alpha
(-i)^\alpha
\exp\left(\frac{\pi i\alpha}{12k}f(h,k)\right).
\]
By Lemma~\ref{lem:f-prime-power},
\[
f(h,k)\equiv
3\cdot 3^{-1}
\left(h+h^{-1}+k(k+3)h\right)
\pmod{24k},
\]
where \(3^{-1}\) is taken modulo \(8k\). Hence
\[
\exp\left(\frac{\pi i\alpha}{12k}f(h,k)\right)
=
\exp\left(
\frac{2\pi i}{8k}
3^{-1}\alpha
\left(h+h^{-1}+k(k+3)h\right)
\right).
\]
Substituting this into \eqref{eq:Ak-omega}, and writing the original
exponential with denominator \(8k\), gives
\[
A_k^{(\alpha)}(n,m)
=
(-i)^\alpha
\sum_{\substack{h\bmod k\\(h,k)=1}}
\left(\frac{-k}{h}\right)^\alpha
\exp\left(
\frac{2\pi i}{8k}
\left(
ah+bh^{-1}
\right)
\right),
\]
where $a=3^{-1}\alpha(k^2+3k+1)-8n$ and $b=3^{-1}\alpha-8m$. Since each reduced residue class modulo
\(k\) has exactly \(8\) reduced lifts modulo \(8k\), we obtain
\[
A_k^{(\alpha)}(n,m)
=
\frac{(-i)^\alpha}{8}
\sum_{\substack{h\bmod 8k\\(h,8k)=1}}
\left(\frac{-k}{h}\right)^\alpha
\exp\left(
\frac{2\pi i}{8k}
\left(
ah+bh^{-1}
\right)
\right).
\]
This is
\[
A_k^{(\alpha)}(n,m)
=
\frac{(-i)^\alpha}{8}S_{\chi_2}(a,b;8k).
\]
For \(p>3\), Lemma~\ref{lem:f-prime-power} gives
\(f(h,k)\equiv 24\cdot 24^{-1}(h+h^{-1})\pmod{24k}\), and no rescaling is needed. This yields the first row of the table. For \(p=3\), the same argument gives a sum naturally modulo \(3k\); each reduced residue class modulo \(k\) has exactly \(3\) reduced lifts modulo \(3k\), giving the factor \(1/3\). This gives the second row.
\end{proof}
\begin{remark}
In Theorem~\ref{prop:Ak-prime-power}, the character \(\chi\) is determined by the parities of \(\alpha\) and \(\lambda\). If \(\alpha\) is even, then \(\chi\) is the trivial character $\mathbf{1}$. If \(\alpha\) is odd, then
\[
\chi(h)=
\begin{cases}
\left(\frac{-h}{p}\right), & p\geq 3 \text{ and } \lambda \text{ is odd},\\[4pt]
\mathbf{1}, & p\geq 3 \text{ and } \lambda \text{ is even},\\[4pt]
\left(\frac{-2}{h}\right), & p=2 \text{ and } \lambda \text{ is odd},\\[4pt]
\left(\frac{-1}{h}\right), & p=2 \text{ and } \lambda \text{ is even}.
\end{cases}
\]
Thus the sums appearing in Theorem~\ref{prop:Ak-prime-power} are either classical Kloosterman sums or quadratic twists. Section~\ref{KloosEval} gives the formulas used to evaluate these sums.
\end{remark}

\section{Evaluation of Kloosterman sums}\label{KloosEval}
This section collects the local Kloosterman-sum evaluations needed in Theorem~\ref{prop:Ak-prime-power}. We use the formulas established in \cite[Sections~3--5]{AJNJ}; see also
\cite[Part~I, Section~5]{derasis2026thesis}.

Let $k=p^\lambda$ be a prime power, and suppose that $\chi$ is defined modulo $p^\beta$, where $\beta\leq\lambda$. Write $\mathbf{1}$ for the trivial character.
\begin{proposition}\label{identity:s_gamma}
	Given integers $a,b$, write $\gcd(a,b,k) = p^\gamma$.
	If $\gamma < \lambda - \val_p(2)$, assume that $\beta \leq \lambda-\gamma$
	and that $\gamma = \val_p(b)$.
	Then
	\begin{equation}
		\label{eqn:selberg}
		S_\chi\left(a,b;k\right)=
		\begin{cases}
			0,
			  &\gamma = \lambda, \, \chi \neq \mathbf{1},\\
			p^\lambda - p^{\lambda -1}
			  &\gamma = \lambda, \,\chi = \mathbf{1},\\
            p^\gamma \cdot \chi(b/p^\gamma) \cdot
			S_\chi\left(ab/p^{2\gamma},1;p^{\lambda-\gamma}\right)
			, &\gamma < \lambda - \val_p(2) .
		\end{cases}
	\end{equation}
	Finally, if $p = 2$ and $\gamma = \lambda-1$ then
	\begin{equation}
		\label{eqn:selberg2}
		S_\chi\left(a,b;k\right) = 
		\begin{cases}
				0,& \chi \neq \mathbf{1},\\
				2^{\lambda-1} \, (-1)^{(a+b)/2^\gamma}, & \chi = \mathbf{1}.
		\end{cases}
	\end{equation}
        
\end{proposition}
\begin{proof}
    See \cite[Proposition~3.3]{AJNJ}.
\end{proof}
For the real-valued characters used below, inversion gives $S_\chi(a,b;k)=S_\chi(b,a;k)$. If $\gamma<\lambda$, then $\gamma=\min\{\val_p(a),\val_p(b)\}$. Interchanging $a$ and $b$ if necessary, we may therefore assume that $\gamma=\val_p(b)$. In the cases considered below, Proposition~\ref{identity:s_gamma} then reduces the remaining evaluations to sums of the form $S_\chi(a,1;k)$. We use this reduction in our implementation.
\begin{definition}
Let $p$ be an odd prime and let $k=p^\lambda$, where $\lambda\geq1$. Define
\[ 
\epsilon_k :=
\begin{cases}
1, & k\equiv1\pmod4,\\
i, & k\equiv3\pmod4.
\end{cases}
\]
We write $\left(\frac{\cdot}{p}\right)$ and $\left(\frac{\cdot}{k}\right)$ for the Legendre and Jacobi symbols, respectively.
\end{definition}
\begin{proposition}
    \label{prop:acuadradop}
    Assume that \(p>2\), and let \(a\) be an integer. Then
    \[
    S_\krop(a,1;k)
    =
    \begin{cases}
        0,
        & p\nmid a
          \text{ and }a\text{ is not a square modulo }k,\\[2mm]
        \displaystyle
        \overline{\delta_k}\epsilon_k
        \legendre{\theta}{p}\legendre{\theta}{k}
        2\sqrt k\,
        \op{Re}\left(
            \delta_k e^{4\pi i \theta/k}
        \right),
        & a\equiv \theta^2\pmod k \text{ and } p\nmid \theta,\\[3mm]
        0,
        & p\mid a \text{ and } \lambda>1,\\[1mm]
        \epsilon_p\sqrt p,
        & p\mid a \text{ and } \lambda=1,
    \end{cases}
    \]
    where
    \[
    \delta_k:=
    \begin{cases}
        1, & pk\equiv1\pmod4,\\
        i, & pk\equiv3\pmod4.
    \end{cases}
    \]
\end{proposition}
\begin{proposition}
    \label{prop:s1-computation}
    Assume that \(p>2\), and let \(a\) be an integer. Except in the case \(\lambda=1\) and \(p\nmid a\), we have
    \[
    S_{\mathbf{1}}(a,1;k)
    =
    \begin{cases}
        0,
        &
        p\nmid a,\quad
        a\text{ is not a square modulo }k
        \text{ and }
        \lambda>1,
        \\[2mm]
        \displaystyle
        \legendre{\theta}{k}\,
        2\sqrt{k}\,
        \op{Re}\left(
            \epsilon_k e^{4\pi i\theta/k}
        \right),
        &
        a\equiv\theta^2\pmod{k},\quad
        p\nmid\theta\text{ and }
        \lambda>1,
        \\[3mm]
        0,
        &
        p\mid a\text{ and }
        \lambda>1,
        \\[1mm]
        -1,
        &
        p\mid a \text{ and }
        \lambda=1.
    \end{cases}
    \]
\end{proposition}
\begin{proof}[Proof of Propositions~\ref{prop:acuadradop} and~\ref{prop:s1-computation}]
    See \cite[Lem.* 2]{Leh38}; alternatively, combine \cite{williams1971} with the identity \(S_\chi(a,1;k) = \overline{\chi}(\theta) \, S_\chi(\theta,\theta;k)\) when $a \equiv \theta^2 \pmod{k}$ with $p \nmid \theta$.
\end{proof}
The remaining case \(\lambda=1\) and \(p\nmid a\) is computed directly; no analogous elementary closed formula is available in general.
\begin{lemma}
Let \(p\) be an odd prime and let \(a\) be an integer. Then
\[
S_{\mathbf 1}(a,1;p)
=
2\sum_{h=1}^{(p-1)/2}
\cos\left(
\frac{2\pi}{p}\left(ah+h^{-1}\right)
\right).
\]
\end{lemma}
For $p=2$, the cases $1\leq\lambda\leq5$ are covered by \cite[Proposition~3.9]{AJNJ}. Recall that, for $\lambda\geq3$, an odd integer is a square modulo $2^\lambda$ if and only if it is congruent to $1$ modulo $8$. The following proposition gives a uniform evaluation for $\lambda\geq6$.
\begin{proposition}
	\label{prop:salie2b}
	Assume that $p=2$, $\lambda\geq 6$, and that $\chi$ has conductor dividing $8$. Then \(S_\chi(a,1;k)=0\) if $a\not\equiv 1\pmod 8$. If $a\equiv 1\pmod 8$, choose an odd integer $\theta$ such that \(a\equiv\theta^2\pmod k\), and set \(\omega=\exp(4\pi i\theta/k)\). Then the following formulas hold.
	\begin{enumerate}
        \item Assume that $\lambda = 2\nu$. Then
	\begin{equation*}
		S_\chi\left(a,1;k\right) =
            2^\nu\overline{\chi}(\theta) \left(		\omega\left(\chi(1)+i^\theta\chi\left(1+2^{\nu-1}\right)\right)
			+
			\overline{\omega}\left(\chi(-1)-i^\theta\chi\left(-1+2^{\nu-1}\right)\right)
			\right)
			.
	\end{equation*}
        \item Assume that $\lambda = 2\nu+1$. Then
	\begin{equation*}
		S_\chi\left(a,1;k\right) =
		2^\nu\overline{\chi}(\theta) \left(
			2\omega i^\theta (\sqrt{i})^{\theta t}\chi\left(1+2^{\nu -1}\right)
			+
		\overline{\omega} \chi\left(-1+2^{\nu-1}\right)
		\left((\sqrt{i})^{\theta t} + i^{-\theta} (\sqrt{i})^{\theta s}\right)
		\right)
		.
    \end{equation*}
    Here $s = 5$ if $\nu = 3$ and $s = 1$ otherwise, 
    and $t = 3$ if $\nu = 3$ and $t = -1$ otherwise.
	\end{enumerate}
\end{proposition}
\begin{proof}
    See \cite[Proposition 3.12]{AJNJ}.
\end{proof}

\section{A recurrence for $p_\alpha(n)$}
\label{sect:recursive}
Before turning to the analytic estimates, we record the exact recurrence used for independent verification of the computed coefficients.
\begin{proposition}
\label{prop:t-colored-sigma-recurrence}
The coefficients $p_\alpha(n)$ satisfy $p_\alpha(0)=1$ and, for every $n\geq1$,
\begin{equation}
\label{eq:t-colored-sigma-recurrence}
p_\alpha(n)
=
\frac{\alpha}{n}
\sum_{k=1}^{n}
\sigma(k)\,p_\alpha(n-k),
\end{equation}
where \(\sigma(k):=\sum_{d\mid k}d\).

\end{proposition}

\begin{proof}
See Gandhi~\cite[(2.7)]{gandhi1963congruences}.
\end{proof}

\begin{remark}
After precomputing the divisor sums $\sigma(1),\ldots,\sigma(N)$, we use the recurrence above to compute $p_\alpha(1),\ldots,p_\alpha(N)$ successively, starting from $p_\alpha(0)=1$. These values provide an independent exact check of the coefficients computed from the Rademacher expansion.
\end{remark}

\section{Upper bounds and truncation error}\label{ordgro}
In this section, we derive an explicit upper bound for $p_\alpha(n)$ to guide the choice of numerical precision. We also recall a truncation-error estimate for the Rademacher-type series, which will be used to choose the summation limits. We begin with the asymptotic formula of Iskander, Jain, and Talvola \cite[Corollary~1.2]{IJT20}.
\begin{proposition}\label{cor:asymptotic-pa}
For all \(\alpha>0\), as \(n\to\infty\), we have
\begin{align}
p_\alpha(n)
\sim
2\pi
\frac{
I_{s_\alpha}\!\left(\frac{\pi\alpha}{6}\lambda_\alpha(n)\right)
}{
\lambda_\alpha(n)^{s_\alpha}
}
\sim
\sqrt{\frac{12}{\alpha}}\,
\frac{
e^{\frac{\alpha\pi}{6}\lambda_\alpha(n)}
}{
\lambda_\alpha(n)^{s_\alpha+\frac{1}{2}}
},
\end{align}
where
\begin{align}
\lambda_\alpha(n):=\sqrt{\frac{24n}{\alpha}-1}.
\end{align}
\end{proposition}
We complement this asymptotic formula with an explicit
upper bound.
\begin{proposition}\label{upperbound}
Let $\alpha>0$ and $n>\alpha/24$. Then 
\[
p_\alpha(n)\le D_\alpha \nu_\alpha(n)^{-\alpha/2} \exp\left( 2\pi \nu_\alpha(n) \sqrt{\frac{\alpha}{6}}\right),
\]
where
\[
D_\alpha = 8\pi^2 \left(1+\frac{e^{1/4}}{2^{s_\alpha}\Gamma\left(s_\alpha+1\right)}\left(1+\frac{2}{\alpha}\right)\right) \sum_{m=0}^{M_\alpha}\mu_\alpha(m)^{s_\alpha+1}p_\alpha(m).
\]
\end{proposition}

\begin{proof}
By the Rademacher-type formula,
\[
p_\alpha(n) = \nu_\alpha(n)^{-s_\alpha} \sum_{m=0}^{M_\alpha} \mu_\alpha(m)^{s_\alpha}p_\alpha(m)\sum_{k=1}^{\infty}\frac{2\pi}{k}A_k^{(\alpha)}(n,m)I_{s_\alpha}\left(\frac{4\pi}{k}\nu_\alpha(n)\mu_\alpha(m)\right).
\]
Taking absolute values and using
\[
\left|A_k^{(\alpha)}(n,m)\right|\le k,
\]
we obtain
\[
0\leq p_\alpha(n)\le 2\pi\nu_\alpha(n)^{-s_\alpha}\sum_{m=0}^{M_\alpha}\mu_\alpha(m)^{s_\alpha}p_\alpha(m)\sum_{k=1}^{\infty} I_{s_\alpha}\left(\frac{4\pi}{k}\nu_\alpha(n)\mu_\alpha(m)\right).
\]

Fix $m$ and set $B=4\pi\nu_\alpha(n)\mu_\alpha(m)$. If $B=0$, the corresponding summand vanishes. We may therefore assume $B>0$ and estimate the inner sum separately over $k\leq B$ and $k>B$.

For \(k\le B\), using \(I_\nu(x)\le e^x\) (see
\cite[\S10.37 and (10.32.1)]{NIST:DLMF}), we have
\[
\sum_{k\le B} I_{s_\alpha}\left(\frac{B}{k}\right)\le B e^{B}.
\]

For \(k>B\), the argument satisfies \(B/k<1\). Using the series expansion
in \cite[(10.25.2)]{NIST:DLMF}, we have
\[
I_\nu(x)\le\frac{e^{1/4}}{\Gamma(\nu+1)}\left(\frac{x}{2}\right)^\nu,\qquad 0<x\le 1.
\]
Taking \(\nu=s_\alpha>1\), we obtain
\begin{align*}
\sum_{k>B}
I_{s_\alpha}\left(\frac{B}{k}\right)&\le\frac{e^{1/4}}{2^{s_\alpha}\Gamma\left(s_\alpha+1\right)}B^{s_\alpha}\sum_{k>B}\frac{1}{k^{s_\alpha}}.
\end{align*}
For \(B\ge 1\), the integral test gives
\[
\sum_{k>B}\frac{1}{k^{s_\alpha}}\le B^{-s_\alpha}+\int_B^\infty x^{-s_\alpha}\,dx=B^{-s_\alpha}+\frac{2}{\alpha}B^{-\alpha/2}.
\]
Consequently,
\[
B^{s_\alpha}\sum_{k>B}\frac{1}{k^{s_\alpha}}\le 1+\frac{2}{\alpha}B \le\left(1+\frac{2}{\alpha}\right)B.
\]
For \(0<B<1\), we use instead
\[
\sum_{k>B}\frac{1}{k^{s_\alpha}}\le\sum_{k=1}^{\infty}\frac{1}{k^{s_\alpha}}\le 1+\frac{2}{\alpha}.
\]
Since \(B^{s_\alpha}\le B\) for \(0<B<1\), the two cases imply
\[
B^{s_\alpha}\sum_{k>B}\frac{1}{k^{s_\alpha}}\le\left(1+\frac{2}{\alpha}\right)B.
\]
Therefore
\[
\sum_{k>B}I_{s_\alpha}\left(\frac{B}{k}\right)\le\frac{e^{1/4}}{2^{s_\alpha}\Gamma\left(s_\alpha+1\right)}\left(1+\frac{2}{\alpha}\right)B.
\]
Hence
\[
\sum_{k=1}^{\infty}I_{s_\alpha}\left(\frac{B}{k}\right)\le B e^B+\frac{e^{1/4}}{2^{s_\alpha}\Gamma\left(s_\alpha+1\right)}\left(1+\frac{2}{\alpha}\right)B.
\]
Since \(e^B\ge 1\), this gives
\[
\sum_{k=1}^{\infty} I_{s_\alpha}\left(\frac{B}{k}\right)\le B e^B\left(1+\frac{e^{1/4}}{2^{s_\alpha}\Gamma\left(s_\alpha+1\right)}\left(1+\frac{2}{\alpha}\right)\right).
\]
Using
\[
B = 4\pi\nu_\alpha(n)\mu_\alpha(m) \le 4\pi\nu_\alpha(n)\sqrt{\frac{\alpha}{24}}=2\pi\nu_\alpha(n)\sqrt{\frac{\alpha}{6}},
\]
we get
\[ 
B e^B\le 4\pi\nu_\alpha(n)\mu_\alpha(m) \exp\left(2\pi\nu_\alpha(n)\sqrt{\frac{\alpha}{6}}\right).
\]
Substituting this back, we obtain the desired inequality with
\(D_\alpha\) as claimed.
\end{proof}


We next turn to the error introduced by truncating the Rademacher-type series.
\begin{definition}\label{def:truncated-rademacher}
Let $\alpha>0$, let $n>\alpha/24$ be an integer, and let $\delta>0$. Define the truncated Rademacher sum by
\begin{equation}\label{approximation}
p_\alpha(n;\delta)
:=
\nu_\alpha(n)^{-s_\alpha}\sum_{m=0}^{M_\alpha} \mu_\alpha(m)^{s_\alpha}p_\alpha(m)\sum_{1\leq k<\frac{2\pi}{\delta}\mu_\alpha(m)}\frac{2\pi}{k}A_k^{(\alpha)}(n,m)I_{s_\alpha}\left(\frac{4\pi}{k}\nu_\alpha(n)\mu_\alpha(m)\right).
\end{equation}
\end{definition}
As $\delta\to0^+$, we have $p_\alpha(n;\delta)\to p_\alpha(n)$. The following estimate is due to Iskander, Jain, and Talvola \cite[Theorem~4.1]{IJT20}.

\begin{proposition}\label{tail-bound}
Let \(\alpha>0\), \(0<\delta<2\pi\mu_\alpha(0)\), and
\(n>\alpha/24\). Then
\[
\left|p_\alpha(n)-p_\alpha(n;\delta)\right| < \frac{C_\alpha}{\delta}\frac{I_{s_\alpha}\!\left(2\delta\nu_\alpha(n)\right)}{\nu_\alpha(n)^{s_\alpha}},
\]
where
\[
C_\alpha := 4\pi^2 \left(1+\frac{2}{\alpha}\right)\mu_\alpha(0) \sum_{m=0}^{M_\alpha}\mu_\alpha(m)^{s_\alpha}p_\alpha(m).
\]
\end{proposition}

\begin{remark}
\label{rmk:choice-delta}
Fix a truncation-error tolerance $\epsilon>0$ and
choose $0<\delta<2\pi\mu_\alpha(0)$ so that the bound in Proposition~\ref{tail-bound} is less than $\epsilon$. Since this bound is strictly increasing in $\delta$ and tends to zero as $\delta\to0^+$, an admissible value of $\delta$ always exists and can be found by bisection. To reduce the number of terms, we choose $\delta$ close to the supremum of the admissible values. For this choice of $\delta$, define
\begin{equation}
\label{eqn:Km}
K_m:=\max\left\{0,\,\left\lceil\frac{2\pi\mu_\alpha(m)}{\delta}\right\rceil-1\right\},\qquad 0\leq m\leq M_\alpha.
\end{equation}
The restriction on $k$ in~\eqref{approximation} is then equivalent to $1\leq k\leq K_m$.
\end{remark}

\section{Algorithms}
\label{sect:algorithms}

\SetKwComment{Comment}{/* }{ */}
\SetKw{KwReturn}{return}
\SetKwInput{KwData}{Input}
\SetKwInput{KwResult}{Output}

We now combine the arithmetic reductions of Sections~\ref{multisec}--\ref{KloosEval} with the analytic estimates of Section~\ref{ordgro} to obtain Algorithm~\ref{alg:colored-partitions} for the exact computation of $p_\alpha(n)$. We describe the evaluation of the truncated Hardy--Ramanujan--Rademacher expansion, including the computation of the $\alpha$-Kloosterman sums and the choice of working precision, and conclude with computational benchmarks.

Let $M_{\alpha}$ be as in~\eqref{eq:nu-mu}. For \(0\leq m\leq M_\alpha\) and \(k\geq1\), define
\begin{equation}
\label{eqn:tmk}
t_{m,k} := \mu_\alpha(m)^{s_\alpha}p_\alpha(m)\frac{2\pi}{k}A_k^{(\alpha)}(n,m) I_{s_\alpha}\left(\frac{4\pi}{k}\nu_\alpha(n)\mu_\alpha(m)\right).
\end{equation}

Let \(\delta\) and \(K_m\) be as in Remark \ref{rmk:choice-delta}. Then the approximation in \eqref{approximation} can be written as
\begin{equation}
\label{eqn:palpha-truncated-double-sum}
p_\alpha(n;\delta)=\nu_\alpha(n)^{-s_\alpha}\sum_{m=0}^{M_{\alpha}}\sum_{k=1}^{K_m}t_{m,k}.
\end{equation}

We also denote by
\begin{equation}
\label{eqn:Ldelta}
L_\delta:=\sum_{m=0}^{M_{\alpha}}K_m
\end{equation}
the total number of terms in the finite double sum.
Notice that the cutoff depends on \(m\). Since
\(\mu_\alpha(m)\) decreases with \(m\), the largest cutoff occurs
for \(m=0\).

The computation of \(A_k^{(\alpha)}(n,m)\) adapts the prime-power factorization framework used for the partition function \cite[Algorithm~2]{Joh12}, overpartitions \cite[Algorithm~8.1]{BSCVRS23}, and general eta-quotients \cite[Algorithm~7.1]{AJNJ}. The essential modification required by the full two-parameter family is that the factorization must propagate both \(n\) and \(m\).

\SetKwComment{Comment}{/* }{ */}
\SetKw{KwReturn}{return}
\SetKwInput{KwData}{Input}
\SetKwInput{KwResult}{Output}

\begin{algorithm}[ht]
\caption{Evaluation of $A_k^{(\alpha)}(n,m)$}
\label{alg:Ak}

\BlankLine

\KwData{Integers $\alpha\geq1$, $k\geq1$, and $n,m\in\mathbb Z$}

\KwResult{$A_k^{(\alpha)}(n,m)$, as defined in
\eqref{eq:alpha-kloosterman}}

\BlankLine

Factorize $k=p_1^{\lambda_1}\cdots p_j^{\lambda_j}$\;

$(n_3,m_3,k_3)\gets(n,m,k)$\;
$(s,i)\gets(1,1)$\;

\While{$s\neq0$ and $k_3\neq1$}{
    $(k_1,k_2)\gets
    \left(p_i^{\lambda_i},\,k_3/p_i^{\lambda_i}\right)$
    \label{step:choose}\;

    $(n_1,m_1),(n_2,m_2)\gets$
    Theorem~\ref{thm:multiplicativity}, applied to
    $A_{k_3}^{(\alpha)}(n_3,m_3)$ with $k_3=k_1k_2$
    \label{step:multiplicativity}\;

    $s_1\gets A_{k_1}^{(\alpha)}(n_1,m_1)$, using
    Theorem~\ref{prop:Ak-prime-power} and the formulas
    from Section~\ref{KloosEval}
    \label{step:primepow}\;

    $(n_3,m_3,k_3)\gets(n_2,m_2,k_2)$\;
    $(s,i)\gets(s\cdot s_1,i+1)$\;
}

\KwReturn{$s$}
\end{algorithm}

An implementation of the analogous procedure for overpartitions is available in \cite{code}. A related implementation for general eta-quotients was developed in connection with \cite{AJNJ}, but is not publicly available.

For the numerical evaluation of \(p_{\alpha}(n)\), we adapt Johansson’s termwise-precision strategy \cite[Theorem 4]{Joh12} to the truncated double sum indexed by \((m,k)\). Because the magnitude of a term depends on both indices, we assign an individual working precision to each pair \((m,k)\). To determine these precisions, define
\begin{equation}
\label{eqn:xmk}
x_{m,k}:=\frac{4\pi}{k}\nu_\alpha(n)\mu_\alpha(m).
\end{equation}
Using \(\left|A_k^{(\alpha)}(n,m)\right|\leq k\) and \(I_{s_\alpha}(x)\leq e^x
\quad (x>0),\)
we obtain 
\begin{equation}\label{eqn:Umk-definition}
    |t_{m,k}|\leq 2\pi\,\mu_\alpha(m)^{s_\alpha}p_\alpha(m)e^{x_{m,k}}=:U_{m,k}.
\end{equation}
We therefore define
\begin{equation}
\label{eqn:rmk}
r_{m,k}:=\max\left\{r_{\min},\left\lceil\log_2(4L_\delta)+\max\left\{
0,\,\log_2 U_{m,k}-s_\alpha\log_2\nu_\alpha(n)\right\}\right\rceil+g\right\},
\end{equation}
where \(r_{\min}\) is a fixed minimum precision and \(g\) is a fixed number of guard bits. We take \(r_{m,k}\) as the initial working precision. The numerical evaluation is performed using certified ball arithmetic, increasing the working precision if necessary until it returns an approximation \(\widehat{t}_{m,k}\) satisfying
\begin{equation}\label{eq:certified-term-error}
\left|t_{m,k}-\widehat{t}_{m,k}\right|<2^{-(r_{m,k}-g)}U_{m,k}.
\end{equation}

\begin{lemma}\label{lem:termwise-precision}
Let \(r_{m,k}\) be defined by \eqref{eqn:rmk}, and suppose that
\(\widehat{t}_{m,k}\) satisfies
\eqref{eq:certified-term-error}. Then
\[
\left|t_{m,k}-\widehat{t}_{m,k}\right|<\frac{\nu_\alpha(n)^{s_\alpha}}{4L_\delta}.
\]
\end{lemma}

\begin{proof}
Put \(V:=\nu_\alpha(n)^{s_\alpha}\). By the definition of \(r_{m,k}\),
\[
r_{m,k}-g\geq\log_2(4L_\delta)+\max\left\{0,\log_2\!\left(\frac{U_{m,k}}{V}\right)\right\}.
\]
Consequently,
\[
2^{-(r_{m,k}-g)}U_{m,k}\leq\frac{U_{m,k}}{4L_\delta\max\{1,U_{m,k}/V\}}=\frac{\min\{U_{m,k},V\}}{4L_\delta}\leq\frac{V}{4L_\delta}.
\]
Combining this inequality with
\eqref{eq:certified-term-error} proves the result.
\end{proof}

Since
\[
\log_2 U_{m,k}=\log_2(2\pi)+s_\alpha\log_2\mu_\alpha(m)+\log_2 p_\alpha(m)+x_{m,k}\log_2(e),
\]
the dominant contribution to \(r_{m,k}\) decreases proportionally to
\[
\frac{\nu_\alpha(n)\mu_\alpha(m)}{k}.
\]
Thus, the first terms in each \(k\)-sum are evaluated with higher precision, while the later terms require substantially fewer bits. Combining the termwise error bounds with the truncation estimate gives the following criterion for recovering
$p_\alpha(n)$ exactly.

\begin{proposition}
\label{prop:numerical-error}
Let $\delta$ be chosen so that the bound in
Proposition~\ref{tail-bound} satisfies
\[
\left|p_\alpha(n)-p_\alpha(n;\delta)\right|<\frac14.
\]
For each pair $(m,k)$ occurring in $p_\alpha(n;\delta)$, let
$\widehat t_{m,k}$ be an approximation of $t_{m,k}$ satisfying
\[
\left|t_{m,k}-\widehat t_{m,k}\right|<\frac{\nu_\alpha(n)^{s_\alpha}}{4L_\delta}.
\]
Then
\[
\left|
p_\alpha(n)-\nu_\alpha(n)^{- s_\alpha}\sum_{m=0}^{M_{\alpha}}\sum_{k=1}^{K_m}\widehat t_{m,k}\right|<\frac12.
\]
Consequently, $p_\alpha(n)$ is obtained by rounding the preceding approximation to the nearest integer.
\end{proposition}

The resulting procedure is summarized in Algorithm~\ref{alg:colored-partitions}.
\begin{algorithm}[ht]
\caption{Evaluation of $p_\alpha(n)$}
\label{alg:colored-partitions}

\KwData{Integers $\alpha\geq1$ and $n>\alpha/24$}
\KwResult{$p_\alpha(n)$}

\BlankLine

$M_\alpha\gets\lfloor\alpha/24\rfloor$

Choose \(\delta\) as in Remark~\ref{rmk:choice-delta}, with
\(\epsilon=1/4\).

\For{$m\gets0$ \KwTo $M_\alpha$}{
    $K_m\gets
    \max\left\{
    0,\,
    \left\lceil
    \frac{2\pi\mu_\alpha(m)}{\delta}
    \right\rceil-1
    \right\}$
}

$L_\delta\gets\sum_{m=0}^{M_\alpha}K_m$

Choose the accumulator precision using Proposition~\ref{upperbound}

$P\gets0$

\For{$m\gets0$ \KwTo $M_\alpha$}{
    \For{$k\gets1$ \KwTo $K_m$}{
        $A\gets A_k^{(\alpha)}(n,m)$ using Algorithm~\ref{alg:Ak}

        \If{$A\neq0$}{
            $U_{m,k}\gets$ the bound in
            \eqref{eqn:Umk-definition}

            $r_{m,k}\gets$ the precision in
            \eqref{eqn:rmk}

            $\widehat t_{m,k}\gets$
compute a certified approximation to $t_{m,k}$, starting with precision
$r_{m,k}$ and increasing the precision until~\eqref{eq:certified-term-error} is satisfied

            Convert $\widehat t_{m,k}$ to the accumulator precision

            $P\gets P+\widehat t_{m,k}$
        }
    }
}

$P\gets\nu_\alpha(n)^{-s_\alpha}P$

$P\gets$ round the real part of $P$ to the nearest integer

\KwReturn{$P$}
\end{algorithm}

\begin{remark}
The accumulator used for the finite double sum must have sufficient precision to represent the final value of \(p_\alpha(n)\), together with the required guard bits. Proposition~\ref{upperbound} provides an a priori bound for its magnitude. Individual terms may be evaluated with the smaller precisions \(r_{m,k}\) from \eqref{eqn:rmk}, and are then converted to the precision of the accumulator before summation.
Values with \(n\leq\alpha/24\), as well as other small values of
\(p_\alpha(n)\), are computed using the recursive formula in
Proposition~\ref{prop:t-colored-sigma-recurrence}.
\end{remark}

\subsection{Computational benchmarks}

We implemented the algorithms in SageMath 10.9. All computations were performed on a 2025 MacBook Air with an Apple M4 processor and 16 GB of memory, running macOS Sequoia 15.6; the reported times are elapsed wall-clock times.

For $\alpha=5$ and $n=10^6$, the computation took \(0.17\) seconds and produced a \(2478\)-digit integer satisfying
\[
p_5(10^6)\approx6.697755001499\times10^{2477},
\qquad
p_5(10^6)\equiv117931612608206581\pmod{10^{18}}.
\]
This corrects the numerical value reported in \cite[Example~6.7]{AJNJ}. The discrepancy was caused by an omitted parenthesis in the numerical implementation and by insufficient working precision; neither issue affects the underlying formula or the theoretical results of \cite{AJNJ}. The corrected value was obtained using the certified computation described above.

As a larger test with $\alpha>24$, the computation of
\(p_{100}(10^{10})\) took \(3504.83\) seconds (less than an hour) and
produced a \(1{,}113{,}767\)-digit integer satisfying
\[
p_{100}(10^{10})
\approx 4.233595920454\times10^{1113766},
\qquad
p_{100}(10^{10})
\equiv694713790859603336\pmod{10^{18}}.
\]
The code used for the computations is available in the accompanying repository \cite{mycode}.
\section{New congruences}\label{seccongru}
We now apply our algorithm for computing $p_\alpha(n)$ to the search for new congruences. Ryan, Scherr, Sirolli, and Treneer \cite[Theorems~4.2 and~4.3]{RSST21} give finite criteria for certifying congruences for the coefficients of eta-quotients. We specialize these criteria to colored partitions in Algorithm~\ref{alg:interesting} and use Algorithm~\ref{alg:colored-partitions} to carry out the required coefficient computations. This addresses the computational difficulty noted in \cite[Remark~5.2]{RSST21} and allows us to obtain further explicit examples. 

To apply these criteria, we first express the colored
partition numbers as Fourier coefficients of a modular
form. Let \(\alpha\geq 1\), $t=24/\gcd(\alpha,24)$ and $d=\alpha/\gcd(\alpha,24)$.
Then
\[
\frac{1}{\eta(tz)^{\alpha}}=\sum_{\substack{n\geq -d\\ n\equiv -d\pmod t}}p_{\alpha}\left(\frac{n+d}{t}\right)q^n\in M_{-\alpha/2}^{wh}\left(\Gamma_0(N_{\alpha}),\chi_{\alpha}\right),
\]
where \(N_{\alpha}:=\operatorname{lcm}(4,t^2),
\) and \(\chi_{\alpha}\) is the quadratic Dirichlet character given, for
\(\gcd(n,N_{\alpha})=1\), by
\[
\chi_{\alpha}(n)
=
\begin{cases}
\displaystyle
\left(\frac{(-1)^{\alpha/2}}{n}\right),
& \alpha \text{ even},\\[8pt]
\displaystyle
\left(\frac{12}{n}\right),
& \alpha \text{ odd and }3\nmid\alpha,\\[8pt]
1,
& \alpha \text{ odd and }3\mid\alpha.
\end{cases}
\]

This eta-quotient has weight \(-\alpha/2\). Thus it gives a half-integral weight example when \(\alpha\) is odd and an integral weight example when \(\alpha\) is even. 

We let $\ell > 2$ be a prime. Let us denote
\begin{equation}\label{eqn:kl}
	k_\ell := \begin{cases}
		24, & \ell = 3, \\
		\ell^2-1, & \ell \geq 5.
		\end{cases}
\end{equation}

Furthermore, given a non-zero integer $c$ with $c \mid N_{\alpha} \ell^2$ and $\ell^2
\nmid c$ we denote
\begin{equation}\label{eqn:vanishing_Fp}
	f_{c,\ell} :=
	\frac{N_{\alpha}}{\gcd(c^2,N_{\alpha})} \cdot
	\begin{cases}
		10   & \text{if } \ell = 3 \text{ and } \ell \nmid c, \\
		1   & \text{if }  \ell = 3 \text{ and } \ell \mid c, \\
		\tfrac{\ell^4-1}{24} & \text{if } \ell \geq 5 \text{ and } \ell\nmid c, \\
		\tfrac{\ell^2-1}{24} & \text{if } \ell \geq 5 \text{ and } \ell\mid c.
	\end{cases}
\end{equation}

The repeated occurrence of $N_\alpha$ reflects the level of the underlying eta-quotient.

Assume that $\ell\nmid N_\alpha$ and that
$f_\alpha(z):=\eta(tz)^{-\alpha}$ satisfies Condition~C at $\ell$ with associated value $\varepsilon_\ell$. Condition~C and an algorithm for verifying it for eta-quotients are given in \cite[Definition~2.1 and Algorithm~1]{RSST21}. For \(x\in\mathbb Q\), set
\[
a_\alpha(x):=p_\alpha\!\left(\frac{x+d}{t}\right),
\]
where \(p_\alpha(y)=0\) if \(y\notin\mathbb Z_{\geq0}\). For the integer $\kappa$ computed in Algorithm~\ref{alg:interesting}, define
\[
\mathcal H_Q(n):=
\begin{cases}
a_\alpha(Q^2n)
+\displaystyle\legendre{(-1)^{(\kappa-1)/2}n}{Q}
Q^{(\kappa-3)/2}a_\alpha(n)
+Q^{\kappa-2}a_\alpha(n/Q^2),
& \alpha\ \text{odd},\\[6pt]
a_\alpha(Qn)+Q^{\kappa-1}a_\alpha(n/Q),
& \alpha\ \text{even}.
\end{cases}
\]

Given an integer $j \geq 1$ we say that a prime $Q$ is a \emph{candidate} (for yielding congruences) if $Q \equiv -1 \pmod{N_{\alpha} \ell^j}$. A candidate prime is called \emph{interesting} if it yields the congruence described below.
\begin{algorithm}[ht]
\LinesNumbered

\KwData{Parameters \(\alpha,\ell,j,Q\) satisfying the
hypotheses above.}

\KwResult{\texttt{True} if the finite criterion certifies \(Q\); otherwise \texttt{False}.}

\BlankLine

\(
v_0\gets
\min_s\operatorname{ord}_s(f_\alpha),
\)
where \(s\) runs over the cusps of \(\Gamma_0(N_\alpha)\)\;

\(\beta\gets\) the least integer \(b\geq j-1\) such that
\[
-\alpha+\ell^b k_\ell>0
\quad\text{and}\quad
\ell^b f_{c,\ell}>-v_0
\]
for every cusp \(a/c\) of \(\Gamma_0(N_\alpha\ell^2)\) with
\(\ell^2\nmid c\)\;

\(\kappa\gets-\alpha+\ell^\beta k_\ell\)\;

\(n_0\gets\) the Sturm bound for
\[
S_{\kappa/2}
 \bigl(\Gamma_0(N_\alpha\ell^2),\chi_\alpha\bigr)
\]
given by \cite[Proposition~4.1]{RSST21}\;

\For{\(n\gets1\) \KwTo \(n_0\)}{
    \If{\label{step-imp}\(\ell\nmid n\), \(\legendre{n}{\ell}\neq\varepsilon_\ell\), and
    \(\mathcal H_Q(n)\not\equiv0\pmod{\ell^j}\)}{
        \KwReturn{\texttt{False}}\;
    }
}

\KwReturn{\texttt{True}}\;

\caption{Finite criterion for certifying interesting primes}
\label{alg:interesting}
\end{algorithm}

The following proposition is the specialization of
\cite[Theorems~4.2 and~4.3]{RSST21} to \(\eta(tz)^{-\alpha}\).
\begin{proposition}
\label{prop:interesting}
Let \(\alpha\) be a positive integer and \(\ell\) be an odd prime with
\(\gcd(\ell,N_{\alpha})=1\), let \(j\geq1\), and suppose that
Condition~C is satisfied with parameter
\(\varepsilon_{\ell}\).

Let \(Q\) be a prime satisfying \(Q\equiv-1\pmod{N_{\alpha}\ell^{j}}\) and define
\[
    r_{\alpha}=
    \begin{cases}
        1, & \text{if \(\alpha\) is even},\\
        3, & \text{if \(\alpha\) is odd}.
    \end{cases}
\]

If Algorithm~\ref{alg:interesting} returns \texttt{True} for
\((\alpha,\ell,j,Q)\), then
\[
    p_{\alpha}\!\left(
        \frac{Q^{r_{\alpha}}n+d}{t}
    \right)
    \equiv0\pmod{\ell^{j}},
\]
for every integer \(n\geq1\) satisfying \(\gcd(n,\ell Q)=1\), and \(\left(\frac{-n}{\ell}\right)\neq\varepsilon_{\ell}\).
\end{proposition}

\begin{remark}
The existence results in \cite[Proposition~1.5]{Tre08} imply that, under the hypotheses above, a positive proportion of the relevant candidate primes \(Q\) satisfy the required congruences.
\end{remark}
\begin{remark}
    An output of \texttt{False} means only that the sufficient finite criterion in Algorithm~\ref{alg:interesting} did not certify \(Q\); it does not by itself prove that \(Q\) is not interesting.
\end{remark}

To keep these computations tractable, we select the colors $\alpha$ according to the value of $\gcd(\alpha,24)$. A larger common divisor generally yields smaller values of $t$ and $N_\alpha$, thereby reducing both the level $N_\alpha\ell^2$ of the modular forms in the finite verification criterion and the modulus $N_\alpha\ell^j$ defining the candidate primes $Q$. It also generally results in smaller Sturm bounds. This choice is purely computational and does not impose any additional restriction on the underlying congruence theorem; rather, it allows us to examine more parameter combinations with the available computational resources.

To state our results concisely, we introduce the notation
\[
\mathcal{P}_{m}(B)
:=
\left\{
Q\leq B:
Q\text{ is prime and }Q\equiv-1\pmod m
\right\},
\]
for positive integers $m$ and $B$. In our application, $m=N_\alpha\ell^j$. The following theorem records the non-elementary congruences certified by our implementation. The ordered pairs have the form \((\alpha,\varepsilon_\ell)\), and the colors are grouped according to their common value of \(N_\alpha\). Cases that recover known results or follow directly from standard product identities are discussed afterward and included in the accompanying computational data.

\begin{theorem}
\label{thm:colored-congruences}
Algorithm~\ref{alg:interesting} returns \texttt{True} in each of the following selected cases.

\begin{enumerate}

\item For \((\ell,j)=(5,1)\):
\begin{itemize}
\item $(\alpha,\varepsilon_5)
\in\{(6,1),(18,-1)\},
\qquad
Q\in\mathcal{P}_{80}(2000);$

\item$(\alpha,\varepsilon_5)
\in\{(8,1),(16,-1)\},
\qquad
Q\in\mathcal{P}_{180}(2000);$

\item$(\alpha,\varepsilon_5)=(3,1),
\qquad
Q=1279.$
\end{itemize}

\item For \((\ell,j)=(13,1)\):
\begin{itemize}
\item$(\alpha,\varepsilon_{13})
\in\{(6,1),(18,1)\},
\qquad
Q=1871.$
\end{itemize}

\item For \((\ell,j)=(5,2)\):
\begin{itemize}
\item$(\alpha,\varepsilon_5)=(12,1),
\qquad
Q\in\{599,1499\};$

\item$(\alpha,\varepsilon_5)
\in\{(6,1),(18,-1),(30,1),(42,-1),(54,1)\},
\qquad
Q\in\mathcal{P}_{400}(5000);$

\item$(\alpha,\varepsilon_5)
\in\{(8,1),(16,-1)\},
\qquad
Q=2699.$
\end{itemize}

\end{enumerate}

Consequently, for every parameter combination listed above and every corresponding prime \(Q\), one has
\[
p_{\alpha}\left(
\frac{Q^{r_{\alpha}}n+d}{t}
\right)
\equiv 0\pmod{\ell^j}
\]
for every integer \(n\geq1\) satisfying
\[
\gcd(n,\ell Q)=1
\qquad\text{and}\qquad
\legendre{-n}{\ell}\neq\varepsilon_\ell.
\]
\end{theorem}

The complete search produced additional congruences that are not displayed in Theorem~\ref{thm:colored-congruences}. In particular, our computations recover the \(24\)-colored congruences obtained by Ryan, Scherr, Sirolli, and Treneer~\cite[Proposition~5.3]{RSST21} and extend their catalogue of explicit examples to additional colors and prime-power moduli. Other outputs follow from elementary reductions of the colored-partition generating function, using the freshman's dream congruence together with Euler's pentagonal-number theorem and Jacobi's identity. These include all the cases with \(\ell=3\), the successful computations with \(\ell=7\), and several cases modulo \(5\) and \(13\). The \(24\)-colored congruence modulo \(25\) also follows from the identity recorded by Garvan~\cite[(2.3)]{MR1055413}.

To keep the statement focused, we have included only those congruences that, to the best of the author's knowledge, are not immediate consequences of these identities and have not previously appeared in the literature. The complete computational output, including all recovered and elementary cases, as well as every candidate for which Algorithm~\ref{alg:interesting} returned \texttt{True} or \texttt{False}, is available with the SageMath implementation at \cite{mycode}.

\subsection{A further question}

The multiplicativity and prime-power reduction arguments developed here
rely on the integrality of \(\alpha\), both when exponentiating
congruences for Dedekind sums and when identifying the resulting
multiplier factors with Dirichlet characters. It is therefore natural
to ask whether modified versions of these formulas can be obtained for
non-integral \(\alpha\). Such formulas could make the
Iskander--Jain--Talvola expansion computationally effective for
\(p_\alpha(n)\) throughout the full range of real parameters
\(\alpha>0\).

\nocite{MR434929,MR1027834}
\bibliography{colorbiblio.bib}
\bibliographystyle{alphaurl}

\end{document}